\documentclass{amsart}
\usepackage{amsmath}
\usepackage{braket}
\usepackage{amssymb}
\usepackage[english]{babel}
\usepackage{amsfonts}
\usepackage{enumerate}

\newtheorem{thm}{Theorem}[section]
\newtheorem{prop}[thm]{Proposition}
\newtheorem{lemma}[thm]{Lemma}

\theoremstyle{definition}
\newtheorem{defin}[thm]{Definition}

\newtheorem{example}{Example}

\newcommand\dom{\mathop\mathrm{dom}}
\newcommand\FV{\mathop\text{FV}}
\newcommand\rel{\mathop\text{rel}}
\newcommand\ESO{\mathrm{ESO}}
\newcommand\nmodels\nvDash
\renewcommand\models\vDash

\newcommand\I{\mathrm{I}}
\newcommand\FO{\mathrm{FO}}
\newcommand\Mod{\mathop\mathrm{Mod}}
\newcommand\df{\mathrm{D}}
\newcommand\M{\mathbb{M}}
\def\dep{D}

\renewcommand\varphi\phi

\renewcommand\overline\bar

\title{Characterizing Quantifier Extensions of Dependence Logic} 
\author{Fredrik Engstr\"om \and Juha Kontinen}
\date{\today}

\begin{document}

\begin{abstract} We characterize the expressive power	of extensions of 
	Dependence Logic and Independence Logic  by monotone generalized quantifiers 
	in terms of quantifier extensions of existential second-order logic.
\end{abstract}

\maketitle

\section{Introduction}

We study extensions of dependence logic $\df$  by monotone first-order 
generalized quantifiers. Dependence logic \cite{Vaananen:2007}  extends 
first-order logic by dependence atomic formulas
\begin{equation*}\dep(t_1,\ldots,t_n)
\end{equation*} the meaning of which is that the value of the term $t_n$ is
functionally determined by the values of $t_1,\ldots, t_{n-1}$. While in 
first-order logic
the order of quantifiers solely determines the dependence relations between 
variables, in
dependence logic more general dependencies between variables can be expressed.  
In fact,  dependence logic is equivalent to existential second-order logic 
$\ESO$ in expressive power. Historically dependence logic was preceded by  
partially ordered quantifiers (Henkin quantifiers) of Henkin  
\cite{Henkin:1961} and Independence-Friendly (IF) logic of
Hintikka and Sandu \cite{Hintikka:1989}. 

The framework of dependence logic, so-called team semantics,  has turned out 
be very flexible to allow
interesting generalizations. For example, the extensions of dependence logic 
in terms of  intuitionistic implication and linear implication was  introduced 
in
\cite{Abramsky:2009}. Also new variants of the dependence atoms was introduced 
in \cite{Engstrom:2011},
\cite{Gradel:2011} and \cite{Galliani:2011}, and generalized quantifiers  in  
\cite{Engstrom:2011} and \cite{Kontinen:2011A}.

Engstr\"om, in \cite{Engstrom:2011}, considered extensions of $\df$ in terms 
of first-order generalized quantifiers. The reason for doing so was partly to 
have a logical framework to analyze partially ordered generalized quantifier 
prefixes compositionally. The paper introduces a general schema to extend 
dependence logic with first-order generalized quantifiers.
There are also alternative ways of extending dependence logic with generalized 
quantifiers, as in \cite{Kontinen:2011A}, where a version of the majority 
quantifier for dependence logic is studied. It is shown that dependence logic 
with that majority quantifier leads to a new descriptive complexity 
characterization of
the counting hierarchy. 

In this paper we continue the study of the logics $\df(Q)$ in the framework 
developed in  \cite{Engstrom:2011}. Our main result shows that the logic
$\df(Q)$ is equivalent, for sentences, to  $\ESO(Q)$, i.e., existential 
second-order logic extended with $Q$. 
We also show analogous characterizations for extensions of 
Independence logic $\I(Q)$, a variant of dependence logic introduced in 
\cite{Gradel:2011} and independently in \cite{Engstrom:2011}, by generalized 
quantifiers. At the end of the paper, we  characterize the open formulas of $\I(Q)$. For $\df(Q)$, finding a characterization of the open formulas remains open. 

\section{Preliminaries}

\subsection{Dependence Logic}

In this section we give a brief introduction to dependence logic. For a 
detailed account see \cite{Vaananen:2007}.

The syntax of dependence logic extends the syntax of first-order logic with 
new atomic formulas, the dependence atoms. There is one
dependence atom for each arity. We write the atom expressing that the term 
$t_n$ is uniquely determined by the values of the terms $t_1,\ldots,t_{n-1}$ 
as
$\dep(t_1,\ldots,t_n)$.\footnote{The dependence atom is denoted by 
	$\mathord=(t_1,\ldots,t_n)$ in the original exposition 
	\cite{Vaananen:2007}.} We assume that all formulas of dependence logic are 
written in negation
normal form, i.e., all negations in formulas occur in front of atomic 
formulas. For a vocabulary $\tau$, $\df[\tau]$ denotes the set of 
$\tau$-formulas of dependence logic. 


The set of free variables of a formula is defined as in first-order logic with 
the extra clause that all variables in a dependence atom are free. We denote 
the set of free variables of a formula $\varphi$ by $\FV(\varphi)$.

To define a compositional semantics for dependence logic we use \emph{sets of 
	assignments}, called \emph{teams} instead of single assignments as in 
first-order logic. An assignment
is a function $s: V \to M$ where $V$ is a finite set of variables and
$M$ is the universe under consideration. Given a universe $M$ a team of $M$ is 
a set of assignments for some fixed finite set of variables $V$. If 
$V=\emptyset$ there is only one assignment, the empty assignment, denoted by 
$\epsilon$. Observe that the team of the empty assignment $\set{\epsilon}$ is 
different from the empty team $\emptyset$.

Given an assignment $s: V \to M$ and $a \in M$ let $s[a/x]: V \cup \set{x} \to 
M$ be the assignment:
$$
s[a/x]: y \mapsto  \begin{cases}
s(y)  &\text{ if $y \in V \setminus \set{x}$, and}\\
a  &\text{ if $x=y$.}
\end{cases}
$$
Furthermore, let $X[M/y]$ be the team $$\set{s[a/y] | s \in X, a \in M},$$
and whenever $f: X \to M$, let $X[f/y]$ denote $$\set{s[f(s)/y] | s \in X}.$$
The domain of a non-empty team $X$, denoted $\dom(X)$, is the set of variables 
$V$.
 The interpretation of the term $t$ in the model $\M$ under the assignment $s$ 
 is denoted by $t^{M,s}$. 

The satisfaction relation for dependence logic $\M ,X \models  \varphi$ is now 
defined as follows.  Below, the notation $\M,s \models \varphi$ refers to the 
ordinary satisfaction relation of first-order logic.

\begin{enumerate}
\item For first-order atomic or negated atomic formulas $\psi$: $ \M ,X 
	\models  \psi \text{ iff } \forall s \in X: \M,s \models \psi$.
\item $\M ,X \models  \dep(t_1,\ldots,t_{n+1}) \text{ iff }
\forall s,s' \in X \bigwedge_{1\leq i \leq n}t_i^{\M,s}
=t_i^{\M,s'}  \rightarrow  t_{n+1}^{\M,s} =t_{n+1}^{\M,s'}$
\item $\M ,X\models \lnot \dep(t_1,\ldots,t_{n+1}) \text{ iff } X =
\emptyset $
\item $\M ,X \models  \varphi \land \psi \text{ iff }\M ,X \models \varphi 
	\text{ and } \M ,X \models  \psi$
\item $\M ,X \models  \varphi \lor \psi \text{ iff } \exists Y,Z  \text{ s.t.  
	} X= Y \cup Z, \text{ and both } \M,Y \models  \varphi \text{ and } \M,Z 
	\models \psi$
\item $\M ,X \models  \exists y \varphi \text{ iff } \exists f: X \to M, 
	\text{ such that } \M,{X[f/y]}
\models \varphi $
\item $\M ,X \models  \forall y \varphi \text{ iff }  \M,X[M/y]
\models \varphi.$
\end{enumerate}


We define $M \models \sigma$ for a sentence $\sigma$ to hold if $M , 
{\set{\epsilon}} \models \sigma$.

Let us make some easy remarks. First, every formula is satisfied by the empty 
team. Second, satisfaction is preserved under taking subteams:
\begin{prop}\label{prop.down}
If $\M ,X \models  \varphi$ and $Y \subseteq X$ then $\M ,Y \models  \varphi$.  
\end{prop}
And thirdly, the satisfaction relation is invariant of the values of the 
non-free variables of the formula:
\begin{prop}
$\M,X \models \varphi$ iff $\M,Y\models \varphi$ where $Y=\set{s 
	\upharpoonright \FV(\varphi) |  s \in X}$.
\end{prop}

The satisfaction relation for first-order formulas reduces to ordinary 
satisfaction in the following way.

\begin{prop}\label{foisflat}
 For first-order formulas $\varphi$ and teams $X$, $\mathbb M ,X \models  
 \varphi$ iff
for all $s \in X: \mathbb M ,s \models  \varphi$.
\end{prop}

By formalizing the satisfaction relation of dependence logic in existential 
second order logic we get the following upper bound on the expressive power of 
dependence logic.
For a team $X$ with domain $\set{x_1,\ldots,x_k}$, let $\rel(X) $ be the 
$k$-ary relation $\set{\langle s(x_1),\ldots,s(x_k)\rangle | s \in X}$. 

\begin{prop}\label{prop23}
Let $\tau$ be a vocabulary and $\varphi$ a $\df[\tau]$-formula with free 
variables $x_1,\dots, x_k$. Then
there is a $\tau\cup \set{R}$-sentence $\psi$ of $\ESO$, in which $R$ appears 
only negatively,  such that for all models $ \M$ and teams $X$ with domain 
$\set{x_1,\dotsc, x_k}$:
\[\M,X\models \varphi \iff (\mathbb{M},\rel(X))\models\psi . \]
\end{prop}

For sentences the proposition gives that $\df \leq \ESO$ and in 
\cite{Vaananen:2007} the converse inequality was shown, hence $\df \equiv 
\ESO$.
In \cite{Kontinen:2009} the following theorem was shown, which together with 
Proposition \ref{prop23}  characterizes open $\tau$-formulas of dependence 
logic as the $R$-negative (downwards closed) fragment of $\ESO[\tau \cup 
\set{R}]$.

\begin{thm} \label{thm24}
Let $\tau$ be a signature and  $R$ a  $k$-ary relation symbol such that   
$R\notin \tau$. Then for every  $\tau\cup \set{R}$-sentence $\psi$ of $\ESO$, 
in which $R$ appears only negatively,  there is a
$\tau$-formula $\phi$ of $\df$  with free variables $x_1,\dotsc, x_k$  such 
that,  for all  $\M$ and  $X$  with domain $\set{x_1,\dotsc, x_k}$:
\begin{equation*}\label{neweq}
 \M,X\models \phi \iff (\M,\rel(X))\models\psi\vee \forall\overline{y}\neg 
 R(\overline{y}) .
\end{equation*}
\end{thm}

\subsection{Independence logic}

Independence logic was introduced in \cite{Gradel:2011} and independently in 
\cite{Engstrom:2011} as a variant of dependence logic in which the dependence 
atoms are replaced by independence atoms $\bar x \perp_{\bar z} \bar 
y$.\footnote{In \cite{Engstrom:2011} multivalued dependence atoms were 
	introduced, denoted by $[\bar z \mathord\twoheadrightarrow \bar x | \bar 
	y]$. The semantics are very similar to the independence atoms.}  The 
semantics of these atoms are defined by:
\begin{multline*} M ,X \models  \bar y \perp_{\bar x} \bar z \text{ iff } \\
\forall s,s' \mathord\in X \Bigl( s(\bar x)= s'(\bar x)  \rightarrow \exists 
s_0 \mathord\in X\bigl(s_0(\bar x,\bar y)=s(\bar x,\bar y ) \land s_0(\bar z) 
= s'(\bar z) \bigr)
\Bigr).
\end{multline*}

The dependence atoms can easily be expressed using the independence atoms, 
implying that independence logic contains dependence logic, in fact this 
containment is proper, as seen from the lack of downwards closure.

On the other hand the analogue of Proposition \ref{prop23} holds for 
independence logic, if the restriction of $R$ appearing only negatively is 
removed. Galliani, in \cite{Galliani:2011}, showed that the also the analogue 
of Theorem \ref{thm24} holds with the same modification, i.e., the open 
$\tau$-formulas of independence logic corresponds exactly to $\tau \cup 
\set{R}$-sentences of ESO.

\subsection{D(Q)}

The notion of a generalized quantifier goes back to
Mostowski \cite{Mostowski:1957} and Lindstr\"om \cite{Lindstrom:1966}. In a 
recent
paper \cite{Engstrom:2011} Engstr\"om introduced semantics for generalized
quantifiers in the framework of dependence logic. We will review the
definitions here.

Let $Q$ be a quantifier of type $\langle k \rangle$, meaning
that $Q$ is a class of $\tau$-structures, where the signature $\tau$ has a
single $k$-ary relational symbol.  Also, assume that $Q$ is monotone 
increasing, i.e.,  for every $M$ and every $A \subseteq B \subseteq M^k$, if 
$A \in Q_M$ then also $B \in Q_M$.
An assignment $s$ satisfies a formula
$Q\bar x\, \varphi$ in the structure $\mathbb M$, written $\mathbb M,s \models 
Qx\,\varphi$, if the set $\set{\bar a \in M^k | \mathbb M,s[\bar a /\bar x] 
	\models \varphi}$ is in $Q_M$, where $Q_M= \set{R \subseteq M^k | (M,R) \in 
	Q}$.

In the context of teams we say that a team $X$ satisfies a formula $Q\bar x\, 
\varphi$,
\begin{equation}\label{def_q}
\mathbb M ,X \models  Q\bar x\, \varphi \text{, if there exists } F: X \to Q_M
\text{ such that } \mathbb M,{X[F/\bar x]} \models \varphi,
\end{equation}
where $X[F/\bar x] = \set{s[\bar a /\bar x] | \bar a \in F(s)}$.  Note that 
this definition works well only with monotone (increasing) quantifiers, see 
\cite{Engstrom:2011} for details.

Let $D(Q)$ be dependence logic extended with the generalized quantifier $Q$ 
with semantics as defined in \eqref{def_q}.

The following easy proposition suggests that we indeed have the right truth 
condition for monotone quantifiers:

\begin{prop}\label{prop:4} \begin{enumerate}[(i)]
\item\label{en1} $D(Q)$ is downwards closed.
\item\label{en2} $D(Q)$ is local, in the sense that $\M,X \models \varphi$ iff 
	$\M,(X \mathbin\upharpoonright \FV(\varphi)) \models \varphi$.
\item\label{en3} Viewing $\exists$ and $\forall$ as generalized quantifiers of 
	type $\langle 1 \rangle$, the truth conditions in \eqref{def_q} are 
	equivalent to the truth conditions of dependence logic.
\item\label{en4} For $\FO(Q)$-formulas $\varphi$ and teams $X$, $\mathbb M ,X 
	\models  \varphi$ iff
for all $s \in X: \mathbb M ,s \models  \varphi$.
\item\label{en5} For every $\df(Q)$ formula $\varphi$ we have $\M,\emptyset 
	\models \varphi$.
\end{enumerate}
\end{prop}

The proofs of \eqref{en1}, \eqref{en2}, \eqref{en4}, and \eqref{en5} are easy 
inductions on the construction of $\varphi$, and \eqref{en3} is proved by 
using \eqref{en1}.

\subsection{$\ESO(Q)$}

We denote by $\ESO$ the existential fragment of second-order logic. The 
extension, $\ESO(Q)$,  of  $\ESO$ by a  generalized quantifier $Q$ is defined 
as follows. 
\begin{defin}
The formulas of $\ESO(Q)$ are built up recursively from atomic and negated 
atomic formulas with  conjuction, disjunction, first-order existential and 
universal quantification, $Q$ quantification, and second-order existential  
relational and functional  quantification.
\end{defin}

A quantifier $Q$ is  \emph{definable} in $\ESO$ if $Q$ is the class of models 
of some $\ESO$-sentence $\phi$, i.e., \[Q=\Mod(\phi).\]

Note that if for every $M$, $\emptyset \in Q_M$ and $M \notin Q_M$ then we can 
use $Q$ to simulate the classical negation, and thus full second-order logic 
is contained in $\ESO(Q)$.  However, if we restrict to monotone (increasing) 
quantifiers we get the following result as in first-order logic:

\begin{prop} Let $Q$ be a monotone quantifier. Then $Q$ is definable in $\ESO$ 
	iff $\ESO(Q)\equiv \ESO$.
\end{prop}
\begin{proof} Since the model class $Q$ is trivially axiomatizable in 
	$\ESO(Q)$, non-definability of $Q$ in $\ESO$ implies that $\ESO(Q)> \ESO$.  
	Assume then that  $Q$ is definable in $\ESO$ and let $\set{R}$ be the 
	vocabulary of $Q$, where $R$ is $k$-ary.  By the assumption, there is 
	$\phi\in \ESO$ such that $\Mod(\phi)=Q$. The idea is now to use the sentence 
	$\phi$ as a uniform definition of $Q$ using substitution. The problem is 
	that there might be negative occurrences of $R$ in $\phi$. By using the 
	monotonicity of $Q$, this problem can be avoided. Define $\psi$ as follows:
\[  \exists P ( \phi(P/R)\wedge \forall \overline{x}( 
	P(\overline{x})\rightarrow R(\overline{x})) ) .     \]
By the monotonicity of $Q$, the sentence $\psi$ also defines $Q$ and it  only 
has one positive occurrence of $R$. We can now compositionally translate 
formulas of $\ESO(Q)$ into the logic $\ESO$, the clause for  $Q$ being the 
only non-trivial one:
\[ Q\overline{x}\theta \rightsquigarrow \psi(\theta/R),    \]
where   $\psi(\theta/R)$  arises by substituting the unique subformula 
$R(\overline{x})$ of $\psi$ by $\theta(\overline{x})$.
\end{proof}
The next example shows that it is easy to find monotone quantifiers which are 
not $\ESO$-definable.

\begin{example} Let $S\subseteq \mathbb{N}$. Then the following quantifiers of 
	type $\langle 1 \rangle$  are monotone:
\begin{eqnarray*}
Q_1 &=& \set{ (M,X) : \vert M\vert \textrm{ finite, and } \emptyset \neq 
	X\subseteq M }\\
Q_S & = & \set{(M,X)  : \vert M\vert \in S \textrm{ and } X=M   } \cup 
\set{(M,X)  :  \vert M\vert \notin S \textrm{ and } X\neq \emptyset  }
\end{eqnarray*}
\end{example}
By, compactness of $\ESO$, $Q_1$ is not definable in $\ESO$. Furthermore, for 
only countably many $S$, the quantifier $Q_S$   is $\ESO$-definable. 


\section{The equivalence of $\df(Q)$ and $\ESO(Q)$}

In this section we consider monotone increasing quantifiers $Q$ satisfying two 
non-triviality assumptions:  $(M,\emptyset)\notin Q$ and  $(M,M^k)\in Q$ for 
all $M$.
We show that, for sentences, the logics $\df(Q)$ and $\ESO(Q)$ are equivalent.

\subsection{A normal form for $\ESO(Q)$}\label{S3.1}

\begin{defin}\label{SNF}
A formula of $\ESO(Q)$ is in \emph{normal form} if it is of the form $\exists 
f_1 \ldots \exists f_k\, \varphi$ and $\varphi$ is a $\FO(Q)$-sentence in 
prenex normal form without existential quantifiers.  \end{defin}
Thus an $\ESO(Q)$ formula is in normal form if it can be written as:
$$
\exists f_1\cdots f_n Q_1' x_1\cdots Q_m' x_m\psi,
$$
where $Q'_i\in \set{ Q,\forall}$ and $\psi$ is a quantifier-free formula. 
In order to show that every formula of $\ESO(Q)$ can be transformed into this 
normal form, we need the following lemma.

\begin{lemma}\label{connectives}  Then the following equivalences hold
\begin{itemize}
\item $Q\overline{x}(\psi\vee \phi) \equiv Q\overline{x}\psi\vee \phi, $
\item $Q\overline{x}(\psi\wedge \phi) \equiv Q\overline{x}\psi\wedge \phi, $
\end{itemize}
where the variables $\overline{x}$ do not appear free in $\phi$.
\end{lemma}

\begin{prop}\label{ESO(Q)_NF}
Every sentence of $\ESO(Q)$ can be written in the normal form of Definition 
\ref{SNF}.
\end{prop}
\begin{proof} The claim is proved using induction on $\phi$. The proof is 
	analogous to the corresponding proof for $\ESO$ (see e.g., Lemma 6.12 in 
	\cite{Vaananen:2007}). The cases of conjunction and disjunction are proved 
	using Lemma  \ref{connectives}.  The case corresponding to $Q$ is analogous 
	to the case of the universal quantifier using the observation that a formula 
	of the form  $Q \overline{x} \exists f \varphi$ is equivalent to $\exists g 
	Q \overline{x}\psi$, where $\psi$ arises from  $\varphi$ by replacing  terms 
	$f(t_1,\ldots,t_k)$ by $g(\overline{x},t_1,\ldots, t_k)$.
\end{proof}

\subsection{The main result}
We will first show a compositional translation mapping formulas of $\df(Q)$ 
into sentences of $\ESO(Q)$. This is analogous to the translation from $\df$ 
into $\ESO$ of Proposition \ref{prop23}. 

\begin{prop}\label{D(Q)toESO(Q)}
Let $\tau$ be a vocabulary and $\phi$ a $\df(Q)[\tau]$-formula with free 
variables $x_1,\dots, x_k$. Then
there is a $\tau\cup \set{R}$-sentence $\psi$ of $\ESO(Q)$, in which $R$ 
appears only negatively,  such that for all models $ \mathbb{M}$ and teams $X$ 
with domain $\set{x_1,\dotsc, x_k}$:
\[\mathbb{M},X\models \phi \iff (\mathbb{M},\rel(X))\models\psi(R) . \]
\end{prop}

\begin{proof} The claim is proved using induction on $\phi$. It suffices to 
	define a translation for  $Q \bar y\, \theta(\bar x,\bar y) $, since the 
	other cases are translated analogously to  Proposition \ref{prop23}:
\[ Q \bar y \, \theta \rightsquigarrow \exists P\bigl( \theta^*(P)\wedge 
	\forall \bar x (R(\bar x)\rightarrow Q\bar y\,P( \bar x,\bar y))\bigr),   \]  
where $\theta^*$ is the translation for $\theta$ given by the induction 
assumption.  \end{proof}

Next we show that, for sentences, Proposition \ref{D(Q)toESO(Q)} can be 
reversed, and thus the following holds.
\begin{thm}\label{SO->D} $\ESO(Q) \equiv \df(Q)$.
\end{thm}

\begin{proof}

Let $\phi$ be a $\ESO(Q)$-sentence. We show that there is a logically 
equivalent sentence $\psi\in  \df(Q)$.
By Proposition~\ref{ESO(Q)_NF} we may assume that $\phi$ is of the form:
\begin{equation}\label{form}
\exists f_1\cdots f_n Q_1' x_1\cdots Q_m' x_m\psi,
\end{equation}
where $Q_i'\in \set{ \forall, Q}$  and  $\psi$ is  quantifier free. Before 
translating  this sentence into $\df(Q)$, we apply certain reductions to it.  
We transform the quantifier-free part $\psi$ of $\phi$ to satisfy the 
condition
that for each of the function symbols $f_i$  there is a
unique tuple $\overline{x}^i$ of pairwise distinct variables such that all 
occurrences of
$f_i$ in $\psi$ are of the form $f_i(\overline{x}^i)$. In order to
achieve this,  we might have to introduce new existentially quantified 
functions and also
universal first-order quantifiers (as in the proof of  Theorem 3.3 in 
\cite{Kontinen:2011B}), but the
quantifier structure of the sentence \eqref{form} does not change. We will now 
assume that the sentence \eqref{form} has this property.

We will next show how the sentence \eqref{form} can be translated into 
$\df(Q)$. We claim that the following
sentence of $\df(Q)$ is a correct translation for \eqref{form}:
\begin{equation}\label{D(Q)-translation}
Q'_1 x_1\cdots Q'_m x_m \exists y_1\cdots \exists y_n \bigl(\bigwedge _{1\le 
	j\le
n}\dep(\overline{x}^i,y_i)\wedge\theta\bigr),
\end{equation}
where $\theta$ is obtained from $\psi$ by replacing all occurrences of the 
term $f_i(\overline{x}^i)$ by $y_i$.

Let us  show that the sentences \eqref{form} and
\eqref{D(Q)-translation} are logically equivalent. Let $\M$  be a structure 
and
let ${\bf f}_1,\ldots,{\bf f}_n$  interpret the function symbols $f_i$. We 
first show the following auxiliary result:
for all teams $X$ with domain
$\set{x_1,\ldots,x_m }$ the following equivalence holds:
\begin{equation}\label{induction}
(\M, \overline{{\bf f}}),X\models \psi \iff \M,X^* \models\theta,
\end{equation}
where $X^* = X(g_1/y_1)\cdots(g_n/y_n)$,  and the functions $g_i$ are defined  
as follows:
\begin{eqnarray*}
g_i(s)&=& {\bf f}_i(s(\overline{x}^i)),
\end{eqnarray*}
and where $s(\overline{x}^i)$ is the tuple obtained by  pointwise application 
of $s$.  Since $\psi$ and $\theta$ are first-order, by Proposition 
\ref{foisflat}, \eqref{induction} follows from the fact that for each $s\in 
X^*$ it holds that
\begin{equation}\label{newadd}
 (\M, \overline{{\bf f}}),s_i \models \psi \iff  \M,s\models \theta,
\end{equation}
where $s'=s\mathbin\upharpoonright \set{x_1,\ldots,x_m}$.  The claim  is 
proved using induction on the structure of the quantifier-free formula $\psi$.

Let us then show that $\phi$ (see \eqref{form}) and
sentence~\eqref{D(Q)-translation} are logically equivalent. Suppose that 
$\M\models \phi$. Then there are  ${\bf f}_1,\ldots,{\bf f}_n$  such that 
\begin{equation}\label{F1}
	(\M, \overline{{\bf f}})\models Q_1' x_1\cdots Q_m' x_m\psi.   
\end{equation}
Now, by \eqref{F1}, there is a team $X$ arising by evaluating the quantifiers 
$Q'_i$ such that \begin{equation}\label{F2}
	(\M, \overline{{\bf f}}),X\models \psi.   \end{equation}
By \eqref{induction}, and the way  the functions $g_i$ are defined,  we get 
that
\[\M, X^\ast \models \bigwedge _{1\le j\le n}\dep(\overline{x}^i,y_i)\wedge  
	\theta, \]
and that \begin{equation}\label{F3}
\M,X\models \exists y_1\cdots \exists y_n\bigl(\bigwedge _{1\le j\le 
	n}\dep(\overline{x}^i,y_i)\wedge  \theta\bigr).  \end{equation}
Finally, \eqref{F3}  implies that \[ \M\models Q'_1 x_1\cdots Q'_m x_m \exists 
	y_1\cdots \exists y_n (\bigwedge _{1\le j\le
n}\dep(\overline{x}^i,y_i)\wedge\theta).\]
The converse implication is proved by reversing the steps above. Note that 
there is some freedom when choosing the functions ${\bf f}_1,\ldots,{\bf 
	f}_n$, since it is enough to satisfy the equivalence in \eqref{induction}.
\end{proof}

We remark that the theorem also holds for quantifiers satisfying only the 
assumptions that for all $M$, $(M,\emptyset)\notin Q$. This is achieved by a 
small trick:
Let $\phi \in  \ESO(Q)$ be a sentence. Suppose $M$ is such that $(M,M^k) \in 
Q$ then the sentence \eqref{D(Q)-translation}, denoted $\phi^*$ in the 
following, is equivalent to $\phi$ on structures over $M$. However, if $M$ is 
such that $(M,M^k)\notin Q$, then $Q$ is trivially false in structures over 
$M$ and hence $\phi$ is equivalent to $\phi_0\in \ESO$, acquired by replacing 
subformulas headed by $Q$ with $\perp$,   in structures over $M$.

It is easy to show, by induction on $\phi$, that \begin{equation}\label{chain}
 \phi_0 \Rightarrow \phi.
\end{equation}
Let $\phi^*_0\in \df$ be a sentence equivalent to $\phi_0$. Let $\theta$ be 
the following $\df(Q)$ sentence:
\[(Q\bar x \top \land \phi^*) \lor \phi^*_0.
\]
Now, assume that  $(M,M^k) \in Q$, then $\theta$ is equivalent, over $M$, to 
$\phi^* \lor \phi_0^*$. By using the fact that $\phi$ is equivalent to 
$\phi^*$ we can see that whenever $\phi_0^*$ is true $\phi^*$ is also true and 
thus $\theta$ is equivalent, again over $M$, to $\phi$. On the other hand if 
$(M,M^k)\notin Q$ then $\theta$ is equivalent, over $M$, to $\phi_0$ which in 
turn is equivalent to $\phi$.

If we assume  $Q$ only to be monotone (i.e., it may be trivial on some 
universes), we can, by a similar trick as above and using the obvious 
generalization of Proposition \ref{ESO(Q)_NF} to $\ESO(Q_1,\ldots, Q_k)$, 
prove that
$$\ESO(Q,Q^d) \leq \df(Q,Q^d),$$
where $Q^d$ is the dual of $Q$, i.e, $Q^d = \set{(M,A^c) | (M,A) \notin Q}$.
This in turn gives us that for any monotone $Q$:
$$\ESO(Q,Q^d) \equiv \df(Q,Q^d).$$
The logic $\df(Q,Q^d)$ might be considered more natural than $\df(Q)$ since 
$\FO(Q) \leq \df(Q,Q^d)$.

In \cite{Gradel:2011} it is shown that $\I \equiv \ESO$, and hence analogously 
to Proposition \ref{D(Q)toESO(Q)} it follows that $\I(Q) \leq \ESO(Q)$. On the 
other hand, since $\df(Q) \leq \I(Q)$ Theorem \ref{SO->D} implies the 
following.
\begin{thm}\label{IESO}
$\I(Q) \equiv \ESO(Q)$.
\end{thm}

\section{Characterizing  the open formulas}

In this section we note that Theorem  \ref{SO->D} can be generalized to open 
formulas. We assume that the generalized quantifiers are monotone and satisfy 
the same non-triviality conditions as in the previous section.

\begin{thm} Let $\tau$ be a signature and  $R$ a  $k$-ary relation symbol such 
	that  $R\notin \tau$. Then for every  $\tau\cup \set{R}$-sentence $\psi$ of 
	$\ESO(Q)$ there is a
$\tau$-formula $\phi$ of $\I(Q)$  with free variables $\bar z = z_1,\dotsc, 
z_k$  such that,  for all  $\M$ and  $X$  with domain $\set{\bar z}$:
\begin{equation}\label{neweq2}
 \M,X\models \phi \iff (\M,\rel(X))\models\psi\vee \forall\overline{y}\neg 
 R(\overline{y}) .
\end{equation}
\end{thm}
\begin{proof}

	The proof follows the proof of Theorem \ref{SO->D} closely with some 
	additional tweaks.
First we translate the formula $\phi$ into the form
\begin{equation}\label{form2}
	\exists f_1\cdots f_n Q_1' x_1\cdots Q_m' x_m \bigl(\forall \bar w\bigl( 
	R(\bar w) \leftrightarrow f_1(\bar w) = f_2(\bar w)\bigr) \land \psi\bigr),
\end{equation}
where $Q_i'\in \set{ \forall, Q}$ and $\psi$ is a quantifier free formula with 
no occurrence of $R$ and such that all occurrences of $f_i$ is of the form 
$f_i(\bar x^i)$. This is done by using the techniques of Proposition 
\ref{ESO(Q)_NF} and Theorem 6.1 in \cite{Galliani:2011}.

	Instead of translating the formula \eqref{form2} into 
	\eqref{D(Q)-translation} we need to assure that the sets chosen by the 
	quantifier prefix $Q_1'x_1 \ldots Q_m'x_m$ are chosen \emph{uniformly} and 
	not depending on the assignments in the team $X$. In $\I(Q)$ we can do this 
	by adding independence atoms in the following way:
	\begin{equation}
Q'_1 x_1\cdots Q'_m x_m \exists y_1\cdots \exists y_n \bigl(
\bigwedge_{1 \le l \le m} x_l \bot_{\set{x_1,\ldots,x_{l-1}}} \bar z\  \land 
\bigwedge _{1\le i\le
	n} y_i \bot_{\overline{x}^i} y_i) \ \wedge\theta\bigr).
	\end{equation}
Here $\theta$ corresponds to the quantifier free formula in the proof of 
Theorem 6.1 in \cite{Galliani:2011}. Observe that $y \bot_{\bar x} y$. is 
equivalent to the dependence atom $D(\bar x,y)$.

	The rest of the proof goes through as in Theorem \ref{IESO}.
\end{proof}

The same proof cannot prove that $\ESO(Q) \le \df(Q)$. This, and the closely 
related question of whatever we can express slashed and backslashed 
quantifiers in $\df(Q)$ remains open.

\section{Conclusion} 

Our results show that the correspondence between dependence logic and 
independence logic on one hand and $\ESO$ on the other is robust in the sense 
that adding generalized quantifiers will not break the correspondences. 

As discussed in section 3, even if we drop the non-triviality conditions, we 
can prove that for any monotone $Q$:
$$\ESO(Q,Q^d) \equiv \df(Q,Q^d).$$
The dual is used only to express that $\lnot Qx\mathord\perp$, which is 
equivalent to $Q^dx \top$. We leave the question of whether $\ESO(Q) \leq 
\df(Q)$ open for arbitrary monotone quantifiers.


\end{document}